\documentclass[11pt,reqno]{amsart}

\usepackage{geometry} 
\usepackage{amsmath, amsthm, amssymb} 
\usepackage{hyperref}
\usepackage{epigraph}
\usepackage{mathtools}
\usepackage{breqn}
\usepackage{longtable} 
\usepackage{pifont} 
\usepackage{scalerel}
\usepackage{array}
\usepackage{lipsum}
\usepackage{scrextend}
\hypersetup{colorlinks=true,linkcolor=blue, linktocpage}
\usepackage{graphicx} 
\usepackage{mathrsfs} 
\usepackage{color}
\usepackage{verbatim}
\usepackage[dvipsnames]{xcolor}
\usepackage{enumitem}

\usepackage{enumitem}
\setlist[enumerate,1]{label=(\arabic*),ref=\arabic*$^\circ$}

\usepackage{tikz-cd} 
\usepackage{graphicx}

\usepackage{geometry}
\makeatletter
\@addtoreset{equation}{section}
\makeatother

\usepackage[all]{xy}
\usepackage{hyperref}

\title{On abelian extensions of finite abelian subgroups of Cremona groups}
\date{}
\author{Luka Filin}

\date{} 

\newcounter{cthm}

\newtheorem{proposition}[equation]{Proposition}
\newtheorem{thm}[equation]{Theorem}

\newtheorem{corollary}[equation]{Corollary}

\theoremstyle{definition}
\newtheorem{definition}[equation]{Definition}
\newtheorem{remark}[equation]{Remark}

\newtheorem{question}[equation]{Question}
\newtheorem{conjecture}[equation]{Conjecture}

\newtheorem{example}[equation]{Example}
\theoremstyle{example}

\newcommand{\Addresses}{{
  \bigskip
  \footnotesize

  \
    \

\textsc{HSE University, Moscow, Russia}
 \\
  \textit{E-mail: lafilin@edu.hse.ru} \texttt{}
}}
\begin{document}
\maketitle
\begin{abstract}
In this note, we study extension properties of finite abelian subgroups of $\mathrm{Bir}(X)$ where~$X$ is a rational (or rationally connected) variety of dimension at most $4$. We are guided by the following question: is it true that if a finite group $G$ faithfully acts on a rationally connected variety of dimension $n$, then $G$ can faithfully act on a terminal Fano variety of dimension $n$? 
Using algebraic methods, we prove that up to dimension~$4$, abelian extensions of finite abelian subgroups of the Cremona group coincide with direct products of such subgroups, with one exception. This result implies a positive answer to the above question up to dimension~$4$ in the case of finite abelian groups, modulo a conjectural description of finite abelian subgroups of $\mathrm{Bir}(X)$ where $X$ is a rationally connected threefold.
\end{abstract}

\section{Introduction} 

The Cremona group $\mathrm{Cr}_n(\mathbb{C}) = \mathrm{Bir}(\mathbb{P}^n)$ is the group of birational automorphisms of the $n$-dimensional projective space. The number $n$ is called the rank of the Cremona group. In this paper, we are interested in finite subgroups of $\mathrm{Cr}_n(\mathbb{C})$. For $n = 1$, one has $\mathrm{Cr}_1(\mathbb{C}) = \mathrm{PGL}(2, \mathbb{C})$, so this case is elementary. 

The classification of finite subgroups of $\mathrm{Cr}_2(\mathbb{C})$ was obtained in \cite{DI09}. As for finite subgroups of $\mathrm{Cr}_3(\mathbb{C})$, the complete classification seems to be out of reach. There exist results concerning some classes of finite groups, see \cite{Pr09} in the case of simple groups. 
The study of finite abelian subgroups of $\mathrm{Cr}_3(\mathbb{C})$ was initiated in the works \cite{Lo24}, \cite{LPZ25}. We recall the relevant results below. 

In higher dimensions, one can study the group of birational automorphisms $\mathrm{Bir}(X)$ of a rationally connected variety $X$, so that the study of $\mathrm{Cr}_n(\mathbb{C})$ becomes a particular case of a more general problem. We recall some results in this direction.  
The bound for the rank of finite abelian $p$-subgroups of $\mathrm{Bir}(X)$ where $X$ is a rationally connected variety was obtained in \cite{KZh24}. 
In dimension $3$, this bound was previously obtained in a series of works \cite{Pr11}, \cite{Pr14}, \cite{PS17}, \cite{Kuz20},
\cite{Xu20}, \cite{Lo22}. 

For $n\geq 4$, to the best of our knowledge there are no classification results on finite subgroups of $\mathrm{Cr}_n(\mathbb{C})$. The reason is that in higher dimensions the geometry becomes more complicated, and the explicit methods of the minimal model program are not developed yet.  

In this paper, we approach the study of finite abelian subgroups of $\mathrm{Bir}(X)$ where~$X$ is a $4$-dimensional rationally connected variety by purely algebraic methods. Our work is guided by the following question (for the relevant definitions see Section \ref{subsec-mfs}).

\begin{question}[{\cite[Question 1.12]{Lo24})}]
\label{the-question}
Is it true that if a finite group $G$ faithfully acts on a Mori fiber space $f\colon X\to Z$ such that $X$ is rationally connected with $\dim X=n$ and $\dim Z>0$, then $G$ admits a faithful action on a Fano variety with terminal singularities of dimension $n$?  
\end{question}

Roughly speaking, in this question it is asked whether Mori fiber spaces with non-trivial base can have ``more'' symmetries than Fano varieties of the same dimension. 
By virtue of the minimal model program, any finite group which faithfully acts on a rationally connected variety of dimension $n$ (and in particular, any finite subgroup of $\mathrm{Cr}_n(\mathbb{C})$) also faithfully acts on a Mori fiber space $f\colon X\to Z$ with $\dim X=n$. Keeping this in mind, one can reformulate Question \ref{the-question} as follows: is it true that if a finite group $G$ faithfully acts on a rationally connected variety of dimension $n$, then $G$ can faithfully act on a terminal Fano variety of dimension $n$?

The answer to Question \ref{the-question} is not known even in the case $n=2$, despite the existence of the classification of finite subgroups of $\mathrm{Cr}_2(\mathbb{C})$, cf. the discussion in \cite[Section 5.7]{DI09}. 
By \cite[Proposition 1.11]{Lo24} the answer to Question \ref{the-question} is positive for finite abelian groups in the case $n=3$, see Proposition \ref{intro-proposition-ext-1-2} and Proposition~\ref{prop-mfs-pt}. In this paper, we consider this question for $n=4$ in the case of finite abelian groups. 

To formulate our results, we introduce some notation. 

\begin{definition}
Let $\mathcal{A}_n$ be the set of all (isomorphism classes of) finite abelian subgroups of $\mathrm{Cr}_n(\mathbb{C})$. 
Let $\mathcal{B}_n$ be the set of all (isomorphism classes of) finite abelian groups that can faithfully act on rationally connected varieties of dimension $n$. 
\end{definition}

Clearly, $\mathcal{A}_n\subseteq \mathcal{B}_n$ for any $n\geq 1$, and $\mathcal{A}_n=\mathcal{B}_n$ for $n=1,2$.

\begin{proposition} 
\label{prop-a1}
The set $\mathcal{A}_1$ consists of the following groups:
\begin{enumerate}
\item 
$\mathbb{Z}/k$\ for \ $k\geq1$, 
\item
$(\mathbb{Z}/2)^2$.  
\end{enumerate} 
 
\end{proposition} 
 
\begin{thm} [{\cite{Bl07}}]  
\label{prop-a2}
The set $\mathcal{A}_2$ consists of the following groups:
\begin{enumerate}
\item
\label{cremona-plane-abelian-1}
$\mathbb{Z}/k\times \mathbb{Z}/m$\ for\ $k\geq
1, m\geq 1$,
 
\item
\label{cremona-plane-abelian-2}
$\mathbb{Z}/2k\times (\mathbb{Z}/2)^2$\ for\ $k\geq 1$,
 
\item
\label{cremona-plane-abelian-3}
$(\mathbb{Z}/4)^2\times \mathbb{Z}/2$,
 
\item
\label{cremona-plane-abelian-4}
$(\mathbb{Z}/3)^3$,
 
\item
\label{cremona-plane-abelian-5}
$(\mathbb{Z}/2)^4$.
\end{enumerate}
\end{thm}

Let $H, G, K$ be finite abelian groups. 
We say that $G$ is an extension of $K$ by $H$, if there exists an exact sequence
\begin{equation}
\label{intro-exact-sequence} 
0 \rightarrow H \rightarrow G \rightarrow K \rightarrow 0. 
\end{equation}
In this case, we will use the notation $G=H \bullet K$. 
We define two operations on the sets of groups. 
 
\begin{definition}
\label{defin-operations}
Consider two nonempty sets of finite abelian groups $\mathcal{A}$ and $\mathcal{B}$. We define 
\[
\mathcal{A} \times \mathcal{B} = \{\, H \times K \ | \ H \in \mathcal{A},\ K \in \mathcal{B} \,\}, 
\]
 and 
 \[
 \mathcal{A} \bullet \mathcal{B} = \{\, H \bullet K \ | \ H \in \mathcal{A},\ K \in \mathcal{B} \, \}.
 \]
We emphasize that here by $H \bullet K$ we mean all abelian extensions of $K$ by $H$.
Also, put $\mathcal{A} + \mathcal{B} = \mathcal{A} \cup \mathcal{B}$.
\end{definition}

Clearly, $A\times B\subseteq A\bullet B$. 
Also, see Proposition \ref{prop-properties} for the elementary properties of these operations, which include associativity and commutativity. 
The starting point of our research is the following observation.

\begin{proposition}[{\cite[Proposition 3.12]{Lo24}}]
\label{intro-proposition-ext-1-2}
Let $\mathcal{A}_1$ be as in Proposition \ref{prop-a1}, and $\mathcal{A}_2$ be as in Theorem \ref{prop-a2}. Then 
\[
\mathcal{A}_1 \times \mathcal{A}_1=\mathcal{A}_1 \bullet \mathcal{A}_1, \quad \quad \quad \quad \mathcal{A}_1 \times \mathcal{A}_2=\mathcal{A}_1 \bullet \mathcal{A}_2.
\] 
\end{proposition} 

In fact, Proposition \ref{intro-proposition-ext-1-2} has a geometric meaning. As shown in Proposition \ref{prop-mfs-pt}, it provides the positive answer to Question \ref{the-question} for finite abelian groups in the case $n\leq 3$. 
The goal of our work is to generalize Proposition \ref{intro-proposition-ext-1-2}. 

Abelian groups of product type were defined in  \cite{Lo24} as elements of $\mathcal{A}_1\times \mathcal{A}_2$. In particular, a group of product type can faithfully act on a product $\mathbb{P}^1\times S$ where $S$ is a rational surface. Thus, such groups form a rather well-understood class of subgroups in $\mathrm{Cr}_1(\mathbb{C})\times\mathrm{Cr}_2(\mathbb{C})\subseteq\mathrm{Cr}_3(\mathbb{C})$. At the moment, it is not clear whether all finite abelian subgroups of $\mathrm{Cr}_3(\mathbb{C})$ are of product type, see Conjecture \ref{intro-conjecture} below. As shown in \cite{LPZ25}, there exist finite abelian groups in $\mathcal{B}_3$ which are not of product type. 

We introduce the following definition.
\begin{definition} 
The sets of \emph{groups of $n$-th product type} are defined inductively: $\mathcal{PA}_1 =~\mathcal{A}_1$, and 
\[
\mathcal{PA}_n = \sum_{i = 1}^{n-1} \mathcal{A}_i \times \mathcal{A}_{n - i} \quad \quad \quad \text{for}\quad \quad \quad n\geq 2.
\]
Similarly, we define
$\mathcal{PB}_1 =~\mathcal{B}_1=\mathcal{A}_1$, and 
\[
\mathcal{PB}_n = \sum_{i = 1}^{n-1} \mathcal{B}_i \times \mathcal{B}_{n - i} \quad \quad \quad \text{for}\quad \quad \quad n\geq 2.
\]
\end{definition} 

In particular, for any $n\geq 1$ we have 
\[
\mathcal{PA}_n\subseteq \mathcal{A}_n, \quad \quad \quad \quad \mathcal{PB}_n\subseteq \mathcal{B}_n, \quad \quad \quad \quad \mathcal{PA}_n\subseteq \mathcal{PB}_n.
\] 
Also,
\[
\mathcal{PA}_2 = \mathcal{A}_1 \times \mathcal{A}_1 = \mathcal{PB}_2, \quad \quad \quad \quad \mathcal{PA}_3 = \mathcal{A}_1 \times \mathcal{A}_2 = \mathcal{PB}_3, 
\]
\[
\mathcal{PA}_4 = \mathcal{A}_1 \times \mathcal{A}_3 + \mathcal{A}_2 \times \mathcal{A}_2, \quad \quad \quad \quad \mathcal{PB}_4 = \mathcal{B}_1 \times \mathcal{B}_3 + \mathcal{B}_2 \times \mathcal{B}_2.
\] 
If a finite abelian group $G$ acts faithfully on the product $X=X_1\times X_2$ such that $X$ is rationally connected with $\dim X=n$ and $X_i\neq X$, then $G$ belongs to $\mathcal{PB}_n$. 

\begin{remark}
\label{rem-MFS-PT}
If a finite abelian group $G$ faithfully acts on a Mori fiber space $f\colon X\to Z$ such that $X$ is rationally connected with $\dim X=n$ and $\dim Z>0$, then there exists an exact sequence \eqref{intro-exact-sequence} where $K$ faithfully acts on $Z$, and $H$ faithfully acts on the schematic generic fiber of $f$ which is a Fano variety (and hence it is rationally connected) over the function field of $Z$. In particular, both $H$ and $K$ faithfully act on a rationally connected variety of dimension $n$, so that $H\in \mathcal{B}_{n-i}$ (cf. Remark \ref{remark-lefschetz}) and $K\in \mathcal{B}_i$ where $i=\dim Z$, hence $1\leq i\leq n-1$. It follows that $G=H\bullet K$.
\end{remark}

Thus, Question \ref{the-question} can be refined as follows. 

\begin{question}
\label{second-question}
Is it true that $\mathcal{A}_i\bullet \mathcal{A}_{n-i}=\mathcal{A}_i\times \mathcal{A}_{n-i}$ for $n\geq 2$ and $1\leq i\leq n-1$? 
\end{question}

\begin{question}
\label{third-question}
Is it true that $\mathcal{B}_i\bullet \mathcal{B}_{n-i}= \mathcal{B}_i\times \mathcal{B}_{n-i}$ for $n\geq 2$ and $1\leq i\leq n-1$? 
\end{question}

According to Proposition \ref{prop-positive-negative}, if the answer to Question \ref{third-question} is positive, then the answer to Question \ref{the-question} is positive in the case of finite abelian groups. 
By Proposition \ref{intro-proposition-ext-1-2}, the answer to Question \ref{second-question} is positive for $n\leq 3$. Our first main result is as follows.

\begin{thm}
\label{intro-main-thm}
$(\mathcal{A}_2 \bullet \mathcal{A}_2)\setminus (\mathcal{A}_2 \times \mathcal{A}_2)=\{ (\mathbb{Z}/4)^5 \}$. 
\end{thm}

In \cite[Example 11.5]{Lo24} it is shown that there exist $4$ groups which belong to $\mathcal{B}_3\setminus \mathcal{PA}_3$. 

\begin{enumerate}
\item
$G_1=(\mathbb{Z}/4)^4$, 
\item
$G_2=(\mathbb{Z}/8)^2\times(\mathbb{Z}/4) \times \mathbb{Z}/2$, 
\item
$G_3=(\mathbb{Z}/6)^2 \times (\mathbb{Z}/3)^2$,
\item
$G_4=(\mathbb{Z}/6)^3 \times \mathbb{Z}/2$. 
\end{enumerate}
 
Define a set 
$
\mathcal{B}'_3=\mathcal{A}_3 + \{G_1,G_2,G_3,G_4\}.
$
One has $\mathcal{B}'_3\subseteq \mathcal{B}_3$.  
Put $\mathcal{A}'_3=\mathcal{PA}_3=\mathcal{A}_1\times \mathcal{A}_2$. One has $\mathcal{A}'_3\subseteq \mathcal{A}_3$.  

\begin{conjecture}[{\cite{LPZ25}}]
\label{intro-conjecture}
$\mathcal{A}_3=\mathcal{A}'_3$, 
$\mathcal{B}_3=\mathcal{B}'_3$. 
\end{conjecture}

Clearly, we have $\mathcal{A}_1 \times \mathcal{A}'_3 = \mathcal{A}_1\times \mathcal{A}_1\times \mathcal{A}_2\subseteq \mathcal{A}_2\times \mathcal{A}_2$, and similarly $\mathcal{A}_1 \bullet \mathcal{A}'_3\subseteq\mathcal{A}_2 \bullet \mathcal{A}_2$ (cf. Proposition \ref{prop-lemma}). 
Thus, Theorem \ref{intro-main-thm} implies the following.

\begin{corollary}
$(\mathcal{A}_1 \bullet \mathcal{A}'_3 + \mathcal{A}_2 \bullet \mathcal{A}_2) \setminus (\mathcal{A}_1 \times \mathcal{A}'_3 + \mathcal{A}_2 \times \mathcal{A}_2) = \{ (\mathbb{Z}/4)^5 \}.$
\end{corollary}

Our second main results is as follows.

\begin{thm}
\label{intro-second-thm}
$\mathcal{B}_1 \bullet \mathcal{B}_3' + \mathcal{B}_2 \bullet \mathcal{B}_2 = \mathcal{B}_1 \times \mathcal{B}_3' + \mathcal{B}_2 \times \mathcal{B}_2$.
\end{thm}
Hence, if Conjecture \ref{intro-conjecture} holds, then the answer to Question \ref{third-question} is positive for $n=~4$. This suggests that in higher dimensions, the class $\mathcal{B}_n$ is better behaved than $\mathcal{A}_n$. Theorem~\ref{intro-second-thm} together with Proposition \ref{prop-positive-negative} implies the following result.

\begin{corollary}
If Conjecture \ref{intro-conjecture} holds, then the answer to Question \ref{the-question} is positive for $n=4$ in the case of finite abelian groups.  
\end{corollary}


\subsection*{Acknowledgements} The author thanks his scientific advisor Konstantin Loginov for posing the problem and encouraging in writing the paper.

\section{Preliminaries}
In this section, we collect some preliminary results on finite abelian groups and Mori fiber spaces. 

\subsection{Extensions of finite abelian groups}

For any finite abelian group $G$, we can write 
$$G = \prod_{p} G_p,$$ 
where $p$ is a prime number and $G_p$ is the $p$-Sylow subgroup of $G$. 
We say that $G_p$ is the $p$-part of $G$. Also we denote $$G_{\neq p} = \prod_{q \neq p} G_q, \quad \quad \text{so that}\quad \quad G = G_p \times G_{\neq p}.$$ 
For $G_p \cong \mathbb{Z}/p^{n_1} \times \mathbb{Z}/p^{n_2} \times \ldots \times \mathbb{Z}/p^{n_r}$ we define the type of $G_p$ to be the vector $[n_1, n_2, \ldots, n_r]$ where $n_1 \ge n_2 \ge \ldots \ge n_r \ge 1.$
 
\begin{remark}
\label{rem-gp}
A sequence of finite abelian groups
\begin{equation}
\label{ex-HGK-exact}
  0 \rightarrow H \rightarrow G \rightarrow K \rightarrow 0
\end{equation} 
is exact if and only if for any prime number $p$ the $p$-parts $H_p, G_p, K_p$ of groups $H, G, K$, respectively, form an exact sequence
\begin{equation}
\label{ex-HGK-exact-p}
  0 \rightarrow H_p \rightarrow G_p \rightarrow K_p \rightarrow 0. 
\end{equation} 
In this case we say that \eqref{ex-HGK-exact-p} is the $p$-part of \eqref{ex-HGK-exact}. 
\end{remark}

For any type $\lambda = [\lambda_1, \ldots, \lambda_k]$ we associate the Young diagram with $k$ rows and $\lambda_i$ squares in the $i$-th row. For two Young diagrams $\lambda = [\lambda_1, \ldots, \lambda_k]$ and $\nu = [\nu_1, \ldots, \nu_q]$, one can define
their product $\lambda \cdot \nu$ as a formal linear combination of Young diagrams with non-negative coefficients, which is Littlewood–Richardson coefficient $c_{\lambda\nu}^{\mu}$ for $\mu = [\mu_1, \ldots, \mu_m]$, see [Fu1, §5.2].
 
\begin{proposition}[{\cite[Section 2]{Fu00}}] 
\label{prop-fulton}
Let $G_p$, $H_p$, and $K_p$ be finite abelian $p$-groups. 
Assume that $G_p$ has type $\mu = [\mu_1, \ldots, \mu_m]$, $H_p$ has type $\lambda = [\lambda_1, \ldots, \lambda_k]$, and $K_p$ has type $\nu = [\nu_1, \ldots, \nu_q]$. Then an extension $G_p=H_p \bullet K_p$ exists if and only if $c_{\lambda\nu}^{\mu} > 0.$
\end{proposition}
 
\begin{example} Let $H = H_2 = \mathbb{Z}/4 \times \mathbb{Z}/2, K = K_2 = \mathbb{Z}/2 \times \mathbb{Z}/2 = (\mathbb{Z}/2)^2$. Then the product of Young diagrams: $$[2, 1] \cdot [1, 1] = [3, 2] + [3, 1, 1] + [2, 2, 1] + [2, 1, 1, 1].$$ Thus, $\{ H \} \bullet \{ K \} = \{ \, \mathbb{Z}/8 \times \mathbb{Z}/4, \mathbb{Z}/8 \times (\mathbb{Z}/2)^2, (\mathbb{Z}/4)^2 \times \mathbb{Z}/2, \mathbb{Z}/4 \times (\mathbb{Z}/2)^3 \, \}.$
\end{example} 
 
\begin{definition}
We denote by $\mathfrak{r}(G_p)$ the rank of $G$, that is, the minimum number of generators of $G$.
\end{definition}

\begin{remark}
\label{rem-rank-groups}
If $G$ is a finite abelian group and $\mathfrak{r}(G)\leq n$ then $G$ belongs to $\mathcal{PA}_n\subseteq \mathcal{A}_n$.  
Also, if $\mathfrak{r}(G_p) \le n$ for any prime number $p$, then $\mathfrak{r}(G) \le n$. In particular, 
if $\mathfrak{r}(G_{\neq 2}) \le 3$ and $G_2 = H \times (\mathbb{Z}/2)^2$ where $\mathfrak{r}(H) \le 3$, then $G$ belongs to $\mathcal{PA}_4$.
\end{remark}

We list some elementary properties of the operations introduced in Definition \ref{defin-operations}:
\begin{proposition} 
\label{prop-properties}
In the notation of Definition \ref{defin-operations}, the following holds: 
\begin{enumerate}
\item
$\mathcal{A} \times \mathcal{B} \subseteq \mathcal{A} \bullet \mathcal{B}$,
\item
$\mathcal{A}\times \mathcal{B}=\mathcal{B}\times \mathcal{A}$,
\item
$(\mathcal{A} \times \mathcal{B}) \times \mathcal{C} = \mathcal{A} \times (\mathcal{B} \times \mathcal{C})$,
\item
$\mathcal{A} \bullet \mathcal{B}=\mathcal{B} \bullet \mathcal{A}$,
\item
 $(\mathcal{A} \bullet \mathcal{B}) \bullet \mathcal{C} = \mathcal{A} \bullet (\mathcal{B} \bullet \mathcal{C})$. 
\end{enumerate}
\end{proposition}
\begin{proof}
Properties (1), (2) and (3) are obvious. Properties (4) and (5) follow from the Proposition \ref{prop-fulton} and Remark \ref{rem-gp} together with commutativity and associativity of the product of Young diagrams, cf. {\cite[Chapter 5]{Fu06}}. 
\end{proof}

\subsection{Mori fiber spaces}
We work over the field of complex numbers $\mathbb{C}$ unless stated otherwise. 
For the language of the minimal model program (the MMP for short) we refer to \cite{KM98}.

\label{subsec-mfs}
Let $G$ be a finite group. 
Recall that a normal projective $G$-variety~$X$ is called $G\mathbb{Q}$-factorial, if every $G$-invariant Weil divisor on~$X$ is $\mathbb{Q}$-Cartier. 
A \emph{Mori fiber space} is a $G\mathbb{Q}$-factorial
variety $X$ 
with at worst terminal singularities together with a
$G$-equivariant contraction $f\colon X\to Z$ to a normal variety $Z$ such that $\rho^G(X/Z)=1$
and
$-K_X$ is ample over $Z$. 

The next proposition shows that the answer to Question \ref{the-question} is positive for $n=3$ in the case of finite abelian groups. 

\begin{proposition}
\label{prop-mfs-pt}
If a finite abelian group $G$ faithfully acts on a Mori fiber space $f\colon X \to Z$ with $\dim X\leq 3$ and $\dim Z>0$ then $G$ belongs to $\mathcal{PA}_3$. In particular, $G$ faithfully acts on $S\times \mathbb{P}^1$ where $S$ is a smooth rational surface. 
\end{proposition}
\begin{proof}
According to Remark \ref{rem-MFS-PT}, we have $G\in \mathcal{A}_1\bullet \mathcal{A}_2$. By
 Proposition \ref{intro-proposition-ext-1-2} we have $G\in \mathcal{A}_1\times \mathcal{A}_2=\mathcal{PA}_3$. 
\end{proof}

The next useful remark follows from the Lefschetz principle.

\begin{remark}
\label{remark-lefschetz}
Let $\mathbb{K}$ be a field of characteristic $0$. Assume that a finite group $G$ is isomorphic to a subgroup of $\mathrm{Bir}(X)$ where $X$ is a rational (resp., rationally connected) variety of dimension $n$ defined over $\mathbb{K}$. Then $G$ is isomorphic to a subgroup of $\mathrm{Bir}(X')$ where $X'$ is a rational (resp., rationally connected) variety of dimension $n$ defined over~$\mathbb{C}$. In particular, the classes of groups $\mathcal{A}_n$ and $\mathcal{B}_n$ do not depend on the field $\mathbb{K}$ once it is algebraically closed and of characteristic $0$.
\end{remark}

Now, we generalize Proposition \ref{prop-mfs-pt}. 

\begin{proposition}
\label{prop-positive-negative}
Assume that the answer to Question \ref{third-question} is positive for the fixed natural number $n$. In other words, for any $i$ such that $1\leq i\leq n-1$, one has 
\[
\mathcal{B}_i\bullet \mathcal{B}_{n-i}= \mathcal{B}_i\times \mathcal{B}_{n-i}.
\] 
Then the answer to Question \ref{the-question} in dimension $n$ is positive in the case of finite abelian groups.
\end{proposition}
\begin{proof}
We use induction on $n$, the case $n=1$ being trivial. Assume that the claim holds for any $k\leq n-1$. 

Assume that $G$ faithfully acts on a Mori fiber space $f\colon X\to Z$ such that $X$ is rationally connected with $\dim X=n$ and $\dim Z=i>0$. Then there is an exact sequence 
\[
0\to H\to G\to K\to 0
\]
where $K\subseteq\mathrm{Aut}(Z)$, $H\subseteq\mathrm{Aut}(X_\eta)$ where $X_\eta$ is the schematic fiber of $f$. Thus, $K\in \mathcal{B}_{i}$ and $H\in \mathcal{B}_{n-i}$ (cf. Remark \ref{remark-lefschetz}). Thus $G=H\bullet K\in \mathcal{B}_i\bullet \mathcal{B}_{n-i}$. 
By assumption, we have that $G=H'\times K'$ where $H'\in \mathcal{B}_j, K'\in \mathcal{B}_{n-j}$ for some $1\leq j\leq n-1$. 

Running $H'$-MMP, we may assume that $H'$ faithfully acts on a Mori fiber space $f_1\colon X_1\to Z_1$. Similarly, running $K'$-MMP, we may assume that $K'$ faithfully acts on a Mori fiber space $f_2\colon X_2\to Z_2$. 
By inductive assumption, $H'$ faithfully acts on a terminal Fano variety $X'_1$ of dimension $j$,  $H'$ faithfully acts on a terminal Fano variety $X_2$ of dimension $n-j$. Hence, $G$ faithfully acts on a terminal Fano variety $X=X_1\times X_2$ of dimension $n$. This proves the claim.  
\end{proof}

\section{Computing extensions}

This section is devoted to the proofs of Theorem \ref{intro-main-thm} and Theorem \ref{intro-second-thm}. We use the notation as in the introduction. 

\begin{proposition}
\label{prop-lemma}
Put $\mathcal{A}'_3=\mathcal{PA}_3=\mathcal{A}_1\times\mathcal{A}_2$. 
Then
\begin{enumerate}
\item
$\mathcal{A}_1 \times \mathcal{A}'_3 \subseteq \mathcal{A}_2 \times \mathcal{A}_2$,
\item
$\mathcal{A}_1 \bullet \mathcal{A}'_3 \subseteq \mathcal{A}_2 \bullet \mathcal{A}_2$,
\item
$\mathcal{A}_2 \times \mathcal{A}_2 \setminus \mathcal{A}_1 \times \mathcal{A}'_3 = \{(\mathbb{Z}/3)^6, (\mathbb{Z}/4)^4\times (\mathbb{Z}/2)^2 \}.$
\end{enumerate}
\end{proposition}
 
\begin{proof} 
By Proposition \ref{prop-properties}, we have
\[
\mathcal{A}_1 \times \mathcal{A}'_3 = \mathcal{A}_1 \times (\mathcal{A}_1 \times \mathcal{A}_2) = (\mathcal{A}_1 \times \mathcal{A}_1) \times \mathcal{A}_2 \subseteq \mathcal{A}_2 \times \mathcal{A}_2.
\]

For the second claim write (use Proposition \ref{intro-proposition-ext-1-2} for the second equality):
\[
\mathcal{A}_1 \bullet \mathcal{A}'_3 = \mathcal{A}_1 \bullet (\mathcal{A}_1 \times \mathcal{A}_2) =  \mathcal{A}_1 \bullet (\mathcal{A}_1 \bullet \mathcal{A}_2) = (\mathcal{A}_1 \bullet \mathcal{A}_1) \bullet \mathcal{A}_2 \subseteq \mathcal{A}_2 \bullet \mathcal{A}_2.
\]

We prove the third claim. We have $\mathcal{A}_1 \times \mathcal{A}'_3= \mathcal{A}_1 \times \mathcal{A}_1 \times \mathcal{A}_2$.
Note that 
\[
\mathcal{A}_1 \times \mathcal{A}_1 = \{ \mathbb{Z}/n \times \mathbb{Z}/k, \mathbb{Z}/2m \times (\mathbb{Z}/2)^2, (\mathbb{Z}/2)^4 \}_{n\geq 1,k\geq 1, m \geq 1} \subseteq \mathcal{A}_2,
\] 
and $\mathcal{A}_2\setminus (\mathcal{A}_1 \times \mathcal{A}_1)=\{(\mathbb{Z}/4)^2\times \mathbb{Z}/2,(\mathbb{Z}/3)^3\}$. Thus, we have 
\[
(\mathcal{A}_2\times \mathcal{A}_2)\setminus (\mathcal{A}_1 \times \mathcal{A}_1\times \mathcal{A}_2)\subseteq (\mathcal{A}_2\setminus \mathcal{A}_1 \times \mathcal{A}_1)\times (\mathcal{A}_2\setminus \mathcal{A}_1 \times \mathcal{A}_1).
\]
Hence $(\mathcal{A}_2\times \mathcal{A}_2)\setminus (\mathcal{A}_1 \times \mathcal{A}_1\times \mathcal{A}_2)=\{(\mathbb{Z}/4)^4\times (\mathbb{Z}/2)^2,(\mathbb{Z}/3)^6\}$, and the claim follows. 
\end{proof}

Now we are ready to prove the first of our main results.
 
\begin{thm}[=Theorem \ref{intro-main-thm}]
\label{thm-2}
Put $\mathcal{A}'_3=\mathcal{PA}_3=\mathcal{A}_1\times\mathcal{A}_2$. Then 
\[
(\mathcal{A}_1 \bullet \mathcal{A}'_3 + \mathcal{A}_2 \bullet \mathcal{A}_2) \setminus (\mathcal{A}_1 \times \mathcal{A}'_3 + \mathcal{A}_2 \times \mathcal{A}_2) = \mathcal{A}_2 \bullet \mathcal{A}_2\setminus \mathcal{A}_2 \times \mathcal{A}_2= \{ (\mathbb{Z}/4)^5 \}.
\]
\end{thm}
\begin{proof}  
The first equality follows from Proposition \ref{prop-lemma}. Hence it is enough to prove the second equality. Put $\mathcal{PA}'_4=\mathcal{A}_1 \times \mathcal{A}'_3 + \mathcal{A}_2 \times \mathcal{A}_2$.

Consider a group $G\in \mathcal{A}_2 \bullet \mathcal{A}_2$. We will show that either $G$ belongs to $\mathcal{PA}'_4$, or $G=(\mathbb{Z}/4)^5$. 
We have an exact sequence 
\begin{equation}
\label{exact-sequence-product-type-2}
0 \rightarrow H \rightarrow G \rightarrow K \rightarrow 0,
\end{equation}
where $H, K \in \mathcal{A}_2$. 
We will consider all possibilities for groups $H$ and $K$. 
The notation $(n \bullet k)$ means that we consider the set $H \bullet K$ where $H$ (resp., $K$) has the number $n$ (resp., $k$) in the list $\mathcal{A}_2$. 
For simplicity, in the formal linear combination of Young diagrams we put $c_{\lambda\nu}^{\mu} = 1$ if $c_{\lambda\nu}^{\mu} > 0$. Since the operation $\bullet$ is commutative (cf. Proposition \ref{prop-properties}), it is enough to consider the case when $n\leq k$. 

\begin{enumerate}
\item
 $(1\bullet 1)$ 
Since $\mathfrak{r}(H) \le 2$ and $\mathfrak{r}(K) \le 2$, we have $\mathfrak{r}(G) \le 4$. It follows that $G$ belongs to $\mathcal{A}_2\times \mathcal{A}_2\subset \mathcal{PA}'_4$. 
 
\item
$(1\bullet 2)$ Consider the $2$-parts of $H$ and $K$. Put $H_2 = \mathbb{Z}/2^{\lambda_1} \times \mathbb{Z}/2^{\lambda_2}, K_2 = \mathbb{Z}/2^{\nu_1 + 1} \times (\mathbb{Z}/2)^2$, for $\lambda_1 \ge \lambda_2$, $\nu_1\geq0$. We may assume that $\lambda_2 \ge 1$ and $\nu_1\geq 1$, otherwise $\mathfrak{r}(G) \le 4$, and hence $G$ belongs to $\mathcal{A}_2\times \mathcal{A}_2$. 
Compute the product of Young diagrams: \[
[\lambda_1, \lambda_2] \cdot [\nu_1 + 1, 1, 1] = \sum [x_1, x_2, x_3] + \sum [x_1, x_2, x_3, 1] + \sum [x_1, x_2, x_3, 1, 1]\] 
where the sums are for some $x_1 \ge \lambda_1, x_2 \ge \lambda_2, x_3 \ge 1$. Since $\mathfrak{r}(G_{\neq 2}) \le 3$, it follows that $G = G_2 \times G_{\neq 2}$ belongs to $\mathcal{PA}'_4$. 
 
\item
$(1\bullet 3)$ Consider the $2$-part of $H$ and $K$. Put $H_2 = \mathbb{Z}/2^{\lambda_1} \times \mathbb{Z}/2^{\lambda_2}, K_2 = (\mathbb{Z}/4)^2 \times \mathbb{Z}/2$, for $\lambda_1 \ge \lambda_2$. We may assume, that $\lambda_2 \ge 1$, otherwise $\mathfrak{r}(G) \le 4$, hence $G$ belongs to $\mathcal{PA}'_4$. Compute: 
\begin{multline*}
[\lambda_1, \lambda_2] \cdot [2, 2, 1] = [\lambda_1 + 2, \lambda_2 + 2, 1] + [\lambda_1 + 2,  \lambda_2 + 1, 2] + \\ 
+ [\lambda_1 + 2,  \lambda_2 + 1, 1, 1] + [\lambda_1 + 2,  \lambda_2, 2, 1] +
 [\lambda_1 + 1,  \lambda_2 + 2, 2] +\\+ [\lambda_1 + 1,  \lambda_2 + 2, 1, 1] + [\lambda_1 + 1,  \lambda_2 + 1, 2, 1] +\\+ [\lambda_1 + 1,  \lambda_2 + 1, 1, 1, 1] + [\lambda_1 + 1,  \lambda_2, 2, 2] + [\lambda_1 + 1,  \lambda_2, 2, 1, 1] +\\+ [\lambda_1,  \lambda_2 + 2, 2, 1] + [\lambda_1,  \lambda_2 + 1, 2, 2] + [\lambda_1,  \lambda_2 + 1, 2, 1, 1] + [\lambda_1,  \lambda_2, 2, 2, 1].
\end{multline*}
 Since $\mathfrak{r}(G_{\neq 2}) \le 2$, it follows that $G = G_2 \times G_{\neq 2}$ belongs to $\mathcal{PA}'_4$.

\item
$(1\bullet 4)$ Consider the product of Young diagram for the $3$-parts of $H$ and $K$. Put $H_3 = \mathbb{Z}/3^{\lambda_1} \times \mathbb{Z}/3^{\lambda_2}, K_3 = (\mathbb{Z}/3)^3$, for $\lambda_1 \ge \lambda_2$. We may assume that $\lambda_2 \ge 1$, otherwise $\mathfrak{r}(G) \le 4$, hence $G$ belongs to $\mathcal{PA}'_4$. Compute: 
\[
[\lambda_1, \lambda_2] \cdot [1, 1, 1] = [\lambda_1 + 1, \lambda_2 + 1, 1] + [\lambda_1 + 1, \lambda_2, 1, 1] + [\lambda_1, \lambda_2 + 1, 1, 1] + [\lambda_1, \lambda_2, 1, 1, 1].
\] 
Since $\mathfrak{r}(G_{\neq 3}) \le 2$, it follows that $G = G_3 \times G_{\neq 3}$ belongs to $\mathcal{PA}'_4$. 
 
\item
$(1\bullet 5)$ Consider the product of Young diagram for the $2$-parts of $H$ and $K$. Put $H_2 = \mathbb{Z}/2^{\lambda_1} \times \mathbb{Z}/2^{\lambda_2}, K_2 = (\mathbb{Z}/2)^4$, for $\lambda_1 \ge \lambda_2$. We may assume that $\lambda_2 \ge 1$, otherwise $G$ is $\mathcal{PA}'_4$. Compute: 
\[
[\lambda_1, \lambda_2] \cdot [1, 1, 1, 1] = [\lambda_1 + 1, \lambda_2 + 1, 1, 1] + [\lambda_1 + 1, \lambda_2, 1, 1, 1] + [\lambda_1, \lambda_2 + 1, 1, 1, 1] + [\lambda_1, \lambda_2, 1, 1, 1, 1].
\] 
Since $\mathfrak{r}(G_{\neq 2}) \le 2$, it follows, that $G = G_2 \times G_{\neq 2}$ belongs to $\mathcal{PA}'_4$.
 
\item
$(2\bullet 2)$ Consider the product of Young diagram for the $2$-parts of $H$ and $K$. Put $H_2 = \mathbb{Z}/2^{\lambda_1} \times (\mathbb{Z}/2)^2, K_2 =\mathbb{Z}/2^{\nu_1} \times (\mathbb{Z}/2)^2$. 
We may assume that $\lambda_1\geq 1$ and $\nu_1\geq 1$, otherwise $\mathfrak{r}(G)\leq 4$ and $G$ belongs to $\mathcal{PA}'_4$. 
Compute: 
\begin{multline*}
[\lambda_1 + 1, 1, 1] \cdot [\nu_1 + 1, 1, 1] = \sum[x_1, x_2, 1, 1, 1, 1] +\\+ \sum[x_1, x_2, 1, 1, 1] +  \sum[x_1, x_2, 1, 1] + \sum[x_1, x_2, 2] + \\+ \sum[x_1, x_2, 2, 1, 1] + \sum[x_1, x_2, 2, 1] + \sum[x_1, x_2, 2, 2]
\end{multline*}
where the sum is for some $x_1 \ge \nu_1 + 1, x_2 \ge 1$. Since $\mathfrak{r}(G_{\neq 2}) \le 2$, it follows that $G = G_2 \times G_{\neq 2}$ belongs to $\mathcal{PA}'_4$.
 
\item
$(2\bullet 3)$ Consider the product of Young diagram for the $2$-parts of $H$ and $K$: 
\begin{multline*}
[\lambda_1 + 1, 1, 1] \cdot [2, 2, 1] = [\lambda_1 + 3, 3, 2] + [\lambda_1 + 3, 3, 1, 1] + [\lambda_1 + 3, 2, 2, 1] +\\+ [\lambda_1 + 3, 2, 1, 1, 1] + [\lambda_1 + 2, 3, 3] + [\lambda_1 + 2, 3, 2, 1] +\\+ [\lambda_1 + 2, 3, 1, 1, 1] + [\lambda_1 + 2, 2, 2, 2] + [\lambda_1 + 2, 2, 2, 1, 1] +\\+ [\lambda_1 + 2, 2, 1, 1, 1, 1] + [\lambda_1 + 1, 3, 3, 1] + [\lambda_1 + 1, 3, 2, 2] +\\+ [\lambda_1 + 1, 3, 2, 1, 1] + [\lambda_1 + 1, 2, 2, 2, 1] + [\lambda_1 + 1, 2, 2, 1, 1, 1].
\end{multline*} 
Since $\mathfrak{r}(G_{\neq 2}) \le 1$, it follows that $G = G_2 \times G_{\neq 2}$ belongs to $\mathcal{PA}'_4$. 
 
\item
$(2\bullet 4)$ Consider the product of Young diagram for the $3$-parts of $H$ and $K$: $$ [\lambda_1] \cdot [1, 1, 1] = [\lambda_1 + 1, 1, 1] + [\lambda_1, 1, 1, 1].$$ Since $\mathfrak{r}(G_{\neq 3}) \le 3$, it follows that $G = G_3 \times G_{\neq 3}$ belongs to $\mathcal{PA}'_4$. 
 
\item
$(2 \bullet 5)$ Similar to the previous case, consider the product of Young diagram for the $2$-parts of $H$ and $K$: 
\begin{multline*}
[\lambda_1 + 1, 1, 1] \cdot [1, 1, 1, 1] = [\lambda_1 + 2, 2, 2, 1] + [\lambda_1 + 2, 2, 1, 1, 1] +\\+ [\lambda_1 + 2, 1, 1, 1, 1, 1] + [\lambda_1 + 1, 2, 2, 1, 1] +\\+ [\lambda_1 + 1, 2, 1, 1, 1, 1] + [\lambda_1 + 1, 1, 1, 1, 1, 1, 1]. 
\end{multline*}
Since $\mathfrak{r}(G_{\neq 2}) \le 1$, it follows that $G = G_2 \times G_{\neq 2}$  belongs to $\mathcal{PA}'_4$.
 
\item
$(3 \bullet 3)$ Consider the product of Young diagram for the $2$-parts of $H$ and $K$: 
\begin{multline*}
[2, 2, 1] \cdot [2, 2, 1] = [4, 4, 2] + [4, 4, 1, 1] + [4, 3, 3] +[4, 3, 2, 1] +\\+ [4, 3, 1, 1, 1] + [4, 2, 2, 2] + [4, 2, 2, 1, 1] + [3, 3, 3, 1]+\\ + [3, 3, 2, 2] + [3, 3, 2, 1, 1] + [3, 3, 1, 1, 1, 1] + [3, 2, 2, 2, 1] + \\ + [3, 2, 2, 1, 1, 1] + [2, 2, 2, 2, 2] + [2, 2, 2, 2, 1, 1].
\end{multline*} 
It can be seen that $G = G_2$, and either $G = (\mathbb{Z}/4)^5$, or $G$ belongs to $\mathcal{PA}'_4$. 
 
\item
$(3 \bullet 4)$ Since $\mathfrak{r}(G_{2}) \le 3$ and $\mathfrak{r}(G_{3}) \le 3$, it follows that $G = G_2 \times G_3$ belongs to $\mathcal{PA}'_4$.
 
\item
$(3\bullet 5)$ Consider the product of Young diagram for the $2$-parts of $H$ and $K$: 
\begin{multline*}
[2, 2, 1] \cdot [1, 1, 1, 1] = [3, 3, 2, 1] + [3, 3, 1, 1, 1] + [3, 2, 2, 1, 1] +\\+ [3, 2, 1, 1, 1, 1] + [2, 2, 2, 1, 1, 1] + [2, 2, 1, 1, 1, 1, 1].
\end{multline*}
Since $G = G_2$, it follows that $G$ belongs to $\mathcal{PA}'_4$. 

\item
$(4 \bullet 4)$ Similar to the previous case, consider the product of Young diagram for $3$-part of groups: $$[1, 1, 1] \cdot [1, 1, 1] = [2, 2, 2] + [2, 2, 1, 1] + [2, 1, 1, 1, 1] + [1, 1, 1, 1, 1, 1].$$ Since $G = G_3$, it follows that $G$ belongs to $\mathcal{PA}'_4$. 
 
\item
$(4 \bullet 5)$ Since $\mathfrak{r}(G_{2}) \le 4$ and $\mathfrak{r}(G_{3}) \le 3$, it follows that $G = G_2 \times G_3$ belongs to $\mathcal{PA}'_4$. 
 
\item
$(5 \bullet 5)$ Consider the product of Young diagram for $2$-part of groups: 
\begin{multline*}
[1, 1, 1, 1] \cdot [1, 1, 1, 1] = [2, 2, 2, 2] + [2, 2, 2, 1, 1] +\\+ [2, 2, 1, 1, 1, 1] + [2, 1, 1, 1, 1, 1, 1] + [1, 1, 1, 1, 1, 1, 1, 1].
\end{multline*}
Since $G = G_2$, it follows that $G$ belongs to $\mathcal{PA}'_4$. The proof is completed.
\end{enumerate} 
\end{proof}
 
Put 
\begin{enumerate}
\item
$G_1=(\mathbb{Z}/4)^4$, 
\item
$G_2=(\mathbb{Z}/8)^2\times(\mathbb{Z}/4) \times \mathbb{Z}/2$, 
\item
$G_3=(\mathbb{Z}/6)^2 \times (\mathbb{Z}/3)^2$,
\item
$G_4=(\mathbb{Z}/6)^3 \times \mathbb{Z}/2$. 
\end{enumerate}
Define a set 
$
\mathcal{B}'_3=\mathcal{A}_3 + \{G_1,G_2,G_3,G_4\}.
$
One has $\mathcal{B}'_3\subseteq \mathcal{B}_3$.  
Above we have defined $\mathcal{A}'_3=\mathcal{PA}_3=\mathcal{A}_1\times \mathcal{A}_2$. One has $\mathcal{A}'_3\subseteq \mathcal{A}_3$.  
Similarly to $\mathcal{PA}'_4=\mathcal{A}_1\times \mathcal{A}'_3+\mathcal{A}_2\times \mathcal{A}_2$, we define $\mathcal{PB}'_4=\mathcal{B}_1\times \mathcal{B}'_3+\mathcal{B}_2\times \mathcal{B}_2$.
 

\begin{remark}
We have $(\mathbb{Z}/4)^5 \in \mathcal{PB}_4'$. Also $\mathcal{PA}_4'\subseteq \mathcal{PA}_4\subseteq \mathcal{PB}_4'$.
\end{remark} 
 
 
\begin{thm}[=Theorem \ref{intro-second-thm}]
$\mathcal{B}_1 \bullet \mathcal{B}_3' + \mathcal{B}_2 \bullet \mathcal{B}_2 = \mathcal{B}_1 \times \mathcal{B}_3' + \mathcal{B}_2 \times \mathcal{B}_2$.
\end{thm} 
 
\begin{proof} To prove this theorem, it is enough to  consider $2\times 4 = 8$ more cases from $\mathcal{B}_1 \bullet \mathcal{B}_3'$ that we did not have when proving Theorem \ref{thm-2}. Consider an exact sequence $$0 \rightarrow H \rightarrow G \rightarrow K \rightarrow 0.$$ We will consider various possibilities for groups $H$ and $K$. We may assume that $H \in \mathcal{A}_1=\mathcal{B}_1$, $K \in \mathcal{B}_3'$. In the notation $(n \bullet k)$ the number $n$ means the number of $H$ in $\mathcal{B}_1$, and $k$ is the number $k$ from $\mathcal{B}'_3$. 

\begin{enumerate}
\item
$(1 \bullet 8)$ Consider the $2$-parts of $H$ and $K$. Put $H_2 = \mathbb{Z}/2^{\lambda_1}, K_2 = (\mathbb{Z}/4)^4$. We may assume that $\lambda_1 \ge 1$, otherwise $G$ belongs to $\mathcal{PB}_4'$. The product of Young diagram: $$[\lambda_1] \cdot [2, 2, 2, 2] = [\lambda_1 + 2, 2, 2, 2] + [\lambda_1 + 1, 2, 2, 2, 1] + [\lambda_1, 2, 2, 2, 2].$$
It follows that $G = G_2 \times G_{\neq 2}$ belongs to $\mathcal{PB}_4'$.
 
\item
$(1 \bullet 9)$ Consider the $2$-parts of $H$ and $K$. Put $H_2 = \mathbb{Z}/2^{\lambda_1}, K_2 = (\mathbb{Z}/8)^2\times(\mathbb{Z}/4) \times \mathbb{Z}/2$. We may assume that $\lambda_1 \ge 1$, otherwise $G$ belongs to $\mathcal{PB}_4'$. The product of Young diagram: 
\begin{multline*}
[\lambda_1] \cdot [3, 3, 2, 1] = [\lambda_1 + 3, 3, 2, 1] + [\lambda_1 + 2, 3, 3, 1] + [\lambda_1 + 2, 3, 2, 2] + [\lambda_1 + 2, 3, 2, 1, 1] +\\+ [\lambda_1 + 1, 3, 3, 2] + [\lambda_1 + 1, 3, 3, 1, 1] + [\lambda_1 + 1, 3, 2, 2, 1] + [\lambda_1, 3, 3, 2, 1].
\end{multline*}
It follows that $G = G_2 \times G_{\neq 2}$ belongs to $\mathcal{PB}_4'$.

\item
$(1 \bullet 10)$ Consider the $3$-parts of $H$ and $K$. Put $H_3 = \mathbb{Z}/3^{\lambda_1}, K_3 = (\mathbb{Z}/3)^4$. We may assume that $\lambda_1 \ge 1$, otherwise $G$ belongs to $\mathcal{PB}_4'$. The product of Young diagram: $$[\lambda_1] \cdot [1, 1, 1, 1] = [\lambda_1 + 1, 1, 1, 1] + [\lambda_1, 1, 1, 1, 1].$$ Since $\mathfrak{r}(G_{\neq 3}) \le 3$, it follows that $G = G_2 \times G_{\neq 2}$ belongs to $\mathcal{PB}_4'$. 

\item
$(1 \bullet 11)$ Consider the $2$-parts of $H$ and $K$. Put $H_2 = \mathbb{Z}/2^{\lambda_1}, K_3 = (\mathbb{Z}/2)^4$. We may assume that $\lambda_1 \ge 1$, otherwise $G$ belongs to $\mathcal{PB}_4'$. The product of Young diagram: $$[\lambda_1] \cdot [1, 1, 1, 1] = [\lambda_1 + 1, 1, 1, 1] + [\lambda_1, 1, 1, 1, 1].$$ Since $\mathfrak{r}(G_{\neq 2}) \le 4$, it follows that $G = G_2 \times G_{\neq 2}$ belongs to $\mathcal{PB}_4'$.
 
\item
$(2 \bullet 8)$ Consider the $2$-parts of $H$ and $K$. The product of Young diagram is: $$[1, 1] \cdot [2, 2, 2, 2] = [3, 3, 2, 2] + [3, 2, 2, 2, 1] + [2, 2, 2, 2, 1, 1].$$ Since $G = G_2$, it follows that $G$ belongs to $\mathcal{PB}_4'$.
 
\item
$(2 \bullet 9)$ Consider the $2$-parts of $H$ and $K$. The product of Young diagram is: 
\begin{multline*}
[1, 1] \cdot [3, 3, 2, 1] = [4, 4, 2, 1] + [4, 3, 3, 1] + [4, 3, 2, 2] + [4, 3, 2, 1, 1] + \\ + [3, 3, 3, 2] + [3, 3, 3, 1, 1] + [3, 3, 2, 2, 1] + [3, 3, 2, 1, 1, 1].
\end{multline*} 
Since $G = G_2$, it follows that $G$ belongs to $\mathcal{PB}_4'$.
 
\item
$(2 \bullet 10)$ Consider the $2$-parts of $H$ and $K$.  The product of Young diagram is: $$[1, 1] \cdot [1, 1] = [2, 2] + [2, 1, 1] + [1, 1, 1, 1].$$ Since $G = G_2 \times G_3$ and $\mathfrak{r}(G) \le 4$, it follows that $G$ belongs to $\mathcal{PB}_4'$.
 
\item
$(2 \bullet 11)$ Consider the $2$-parts of $H$ and $K$.  The product of Young diagram is: $$[1, 1] \cdot [1, 1, 1, 1] = [2, 2, 1, 1] + [2, 1, 1, 1, 1] + [1, 1, 1, 1, 1, 1].$$ Since $G = G_2 \times G_3$, it follows that $G$ belongs to $\mathcal{PB}_4'$. The proof is completed.
\end{enumerate}
\end{proof}
 
\section*{Appendix. Sets of finite abelian groups}
 
We list the sets of finite abelian groups that are discussed in this work.  

\ 

\begin{center}
\label{table-1}
\begin{tabular}{ | m{1.3em} | m{5.0cm} | m{4.5cm} | } 
  \hline
  (1) & $\mathbb{Z}/k$ & $k \ge 1$ \\ 
  \hline
  (2) & $(\mathbb{Z}/2)^2$ & \\ 
  \hline
  \end{tabular}

\ 

\emph{Table 1. The set of groups $\mathcal{A}_1 = \mathcal{B}_1$.}

\

\label{table-2}
\begin{tabular}{ | m{1.3em} | m{5.0cm} | m{4.5cm} | } 
  \hline
  (1) & $\mathbb{Z}/k\times \mathbb{Z}/l$ & $k \geq 1, l \ge 1$ \\ 
    \hline
  (2) & $\mathbb{Z}/2k\times (\mathbb{Z}/2)^2$ & $k\geq 1$ \\ 
    \hline
  (3) & $(\mathbb{Z}/4)^2\times \mathbb{Z}/2$ &  \\ 
    \hline
  (4) & $(\mathbb{Z}/3)^3$ & \\ 
    \hline
  (5) & $(\mathbb{Z}/2)^4$ &  \\ 
  \hline
\end{tabular}

\

\emph{Table 2. The set of groups $\mathcal{A}_2 = \mathcal{B}_2$.}
\end{center}

\

 \begin{center}
\label{table-3}
\begin{tabular}{ | m{1.3em} | m{5.0cm} | m{4.5cm} | } 
  \hline
  (1) & $\mathbb{Z}/k\times \mathbb{Z}/l\times
\mathbb{Z}/m$ & $k\geq 1,\ l\geq 1,\ m\geq 1$ \\ 
  \hline
  (2) & $\mathbb{Z}/2k\times(\mathbb{Z}/4)^2\times
\mathbb{Z}/2$ & $k\geq 1$ \\ 
    \hline
  (3) & $\mathbb{Z}/3k\times(\mathbb{Z}/3)^3$ & $k\geq 1$ \\ 
    \hline
  (4) & $\mathbb{Z}/2k\times \mathbb{Z}/2l\times
(\mathbb{Z}/2)^2$ & $k\geq 1,\ l\geq 1$ \\ 
    \hline
  (5) & $\mathbb{Z}/2k\times (\mathbb{Z}/2)^4$ & $k\geq 1$ \\ 
    \hline
  (6) & $(\mathbb{Z}/4)^2\times (\mathbb{Z}/2)^3$ & \\ 
    \hline
  (7) & $(\mathbb{Z}/2)^6$ &  \\ 
    \hline
\end{tabular}

\

\emph{Table 3. The set of groups $\mathcal{A}'_3=\mathcal{PA}_3=\mathcal{A}_1\times \mathcal{A}_2$.}
\end{center}

\

 \begin{center}
\label{table-4}
\begin{tabular}{ | m{1.3em} | m{5.0cm} | m{4.5cm} | } 
  \hline
  (1) & $\mathbb{Z}/k\times \mathbb{Z}/l\times
\mathbb{Z}/m$ & $k\geq 1,\ l\geq 1,\ m\geq 1$ \\ 
  \hline
  (2) & $\mathbb{Z}/2k\times(\mathbb{Z}/4)^2\times
\mathbb{Z}/2$ & $k\geq 1$ \\ 
    \hline
  (3) & $\mathbb{Z}/3k\times(\mathbb{Z}/3)^3$ & $k\geq 1$ \\ 
    \hline
  (4) & $\mathbb{Z}/2k\times \mathbb{Z}/2l\times
(\mathbb{Z}/2)^2$ & $k\geq 1,\ l\geq 1$ \\ 
    \hline
  (5) & $\mathbb{Z}/2k\times (\mathbb{Z}/2)^4$ & $k\geq 1$ \\ 
    \hline
  (6) & $(\mathbb{Z}/4)^2\times (\mathbb{Z}/2)^3$ & \\ 
    \hline
  (7) & $(\mathbb{Z}/2)^6$ &  \\ 
    \hline
  (8) & $(\mathbb{Z}/4)^4$ &  \\ 
    \hline
  (9) & $ (\mathbb{Z}/8)^2\times(\mathbb{Z}/4) \times \mathbb{Z}/2$ &  \\ 
    \hline
  (10) & $(\mathbb{Z}/6)^2 \times (\mathbb{Z}/3)^2$ &  \\ 
    \hline
  (11) & $(\mathbb{Z}/6)^3 \times \mathbb{Z}/2$ &  \\ 
    \hline
\end{tabular}

\

\emph{Table 4. The set of groups $\mathcal{B}'_3.$}
\end{center}

\

\begin{center}
\label{table-5}
\begin{tabular}{ | m{1.3em} | m{5.0cm} | m{5cm} | } 
  \hline
  (1) & $\mathbb{Z}/n\times \mathbb{Z}/k \times \mathbb{Z}/l\times \mathbb{Z}/m$ & $n\geq 1, k\geq 1,\ l\geq 1,\ m\geq 1$ \\ 
  \hline
  (2) & $\mathbb{Z}/n\times \mathbb{Z}/k \times \mathbb{Z}/2l\times (\mathbb{Z}/2)^2$ & $n\geq 1, k\geq 1,\ l\geq 1$ \\ 
    \hline
  (3) & $\mathbb{Z}/n\times \mathbb{Z}/k \times (\mathbb{Z}/4)^2\times \mathbb{Z}/2$ & $n\geq 1, k\geq 1$ \\ 
    \hline
  (4) & $\mathbb{Z}/n\times \mathbb{Z}/k \times (\mathbb{Z}/3)^3$ & $n\geq 1, k\geq 1$ \\ 
    \hline
  (5) & $\mathbb{Z}/n\times \mathbb{Z}/k \times (\mathbb{Z}/2)^4$ & $n\geq 1, k\geq 1$ \\ 
    \hline
  (6) & $\mathbb{Z}/2n \times (\mathbb{Z}/4)^2 \times (\mathbb{Z}/2)^3$ & $n\geq 1$ \\ 
    \hline
  (7) & $\mathbb{Z}/2n\times (\mathbb{Z}/2)^6$ & $n\geq 1$ \\ 
    \hline
  (8) & $(\mathbb{Z}/4)^4\times (\mathbb{Z}/2)^2$ & \\ 
    \hline
  (9) & $(\mathbb{Z}/4)^2\times (\mathbb{Z}/2)^5$ &\\ 
    \hline
  (10) & $(\mathbb{Z}/3)^6$ & \\ 
    \hline
  (11) & $(\mathbb{Z}/2)^8$ & \\ 
    \hline
\end{tabular}

\

\emph{Table 5. The set of groups $\mathcal{PA}'_4=\mathcal{A}_1\times \mathcal{A}'_3+\mathcal{A}_2\times \mathcal{A}_2$.}
\end{center}

\

 \begin{center}
\label{table-6}
\begin{tabular}{ | m{1.3em} | m{5.0cm} | m{5cm} | } 
  \hline
  (1) & $\mathbb{Z}/n\times \mathbb{Z}/k \times \mathbb{Z}/l\times \mathbb{Z}/m$ & $n\geq 1, k\geq 1,\ l\geq 1,\ m\geq 1$ \\ 
  \hline
  (2) & $\mathbb{Z}/n\times \mathbb{Z}/k \times \mathbb{Z}/2l\times (\mathbb{Z}/2)^2$ & $n\geq 1, k\geq 1,\ l\geq 1$ \\ 
    \hline
  (3) & $\mathbb{Z}/n\times \mathbb{Z}/k \times (\mathbb{Z}/4)^2\times \mathbb{Z}/2$ & $n\geq 1, k\geq 1$ \\ 
    \hline
  (4) & $\mathbb{Z}/n\times \mathbb{Z}/k \times (\mathbb{Z}/3)^3$ & $n\geq 1, k\geq 1$ \\ 
    \hline
  (5) & $\mathbb{Z}/n\times \mathbb{Z}/k \times (\mathbb{Z}/2)^4$ & $n\geq 1, k\geq 1$ \\ 
    \hline
  (6) & $\mathbb{Z}/2n \times (\mathbb{Z}/4)^2 \times (\mathbb{Z}/2)^3$ & $n\geq 1$ \\ 
    \hline
  (7) & $\mathbb{Z}/2n\times (\mathbb{Z}/2)^6$ & $n\geq 1$ \\ 
    \hline
  (8) & $(\mathbb{Z}/4)^4\times (\mathbb{Z}/2)^2$ & \\ 
    \hline
  (9) & $(\mathbb{Z}/4)^2\times (\mathbb{Z}/2)^5$ &\\ 
    \hline
  (10) & $(\mathbb{Z}/3)^6$ & \\ 
    \hline
  (11) & $(\mathbb{Z}/2)^8$ & \\ 
    \hline
  (12) & $\mathbb{Z}/n \times (\mathbb{Z}/4)^4$ & $n\geq 1$ \\ 
    \hline
  (13) & $\mathbb{Z}/n \times (\mathbb{Z}/8)^2 \times \mathbb{Z}/4 \times \mathbb{Z}/2$ & $n\geq 1$ \\ 
    \hline
  (14) & $\mathbb{Z}/n \times (\mathbb{Z}/6)^2 \times (\mathbb{Z}/3)^2$ & $n\geq 1$ \\ 
    \hline
  (15) & $\mathbb{Z}/n \times (\mathbb{Z}/6)^3 \times \mathbb{Z}/2$ & $n\geq 1$ \\ 
    \hline
  (16) & $(\mathbb{Z}/8)^2 \times \mathbb{Z}/4 \times (\mathbb{Z}/2)^3$ & \\ 
    \hline
  (17) & $(\mathbb{Z}/6)^3 \times (\mathbb{Z}/2)^3$ & \\ 
    \hline
\end{tabular}

\

\emph{Table 6. The set of groups $\mathcal{PB}'_4=\mathcal{B}_1\times \mathcal{B}'_3+\mathcal{B}_2\times \mathcal{B}_2$.}
\end{center}

 \Addresses
\end{document}